\documentclass[12pt]{amsart}
\usepackage{color}

\usepackage{amssymb}
\ifx\pdfoutput\undefined
\usepackage{graphicx}  
\DeclareGraphicsExtensions{.pstex,.eps}

\else
\usepackage{amsfonts}
\usepackage[pdftex]{graphicx}  \DeclareGraphicsExtensions{.pdf,.mps}

\fi

\newcommand{\RR}{{\mathbb R}}

\newtheorem{thm}{Theorem}
\newtheorem{prop}{Proposition}[section]

\newtheorem{lem}[prop]{Lemma}

\newtheorem{rem}{Remark}[section]
\newtheorem{defn}[prop]{Definition}

\numberwithin{equation}{section}
\numberwithin{prop}{section}

\def\squarebox#1{\hbox to #1{\hfill\vbox to #1{\vfill}}}

\def\R{\mathbb R}
\def\reals{\mathbb R}

\def\reals{{\mathbb R}}

\def\be{\begin{eqnarray*}}
\def\ee{\end{eqnarray*}}
\def\ben{\begin{eqnarray}}
\def\een{\end{eqnarray}}
\usepackage{amsxtra}

\title
[Stable perturbations for a minimal mass soliton]
{A class of stable perturbations for a minimal mass soliton in three dimensional saturated nonlinear Schr\"odinger equations}

\author[J.L. Marzuola]
{Jeremy L. Marzuola}

\address{Applied Mathematics Department, Columbia University \\
200 S. W. Mudd, 500 W. 120th St., New York City, NY 10027, USA}

\begin{document}    
   
\begin{abstract}

In this result, we develop the techniques of \cite{KS1} and \cite{BW} in order to determine a class of stable perturbations for a minimal mass soliton solution of a saturated, focusing nonlinear Schr\"odinger equation
\begin{eqnarray*}
\left\{\begin{array}{c}
i u_t + \Delta u + \beta (|u|^2) u = 0 \\
u(0,x) = u_0 (x),
\end{array} \right.
\end{eqnarray*}
in $\reals^3$.  By projecting into a subspace of the continuous spectrum of $\mathcal{H}$ as in \cite{S1}, \cite{KS1}, we are able to use a contraction mapping similar to that from \cite{BW} in order to show that there exist solutions of the form
\begin{eqnarray*}
e^{i \lambda_{\min} t} (R_{min} + e^{i \mathcal{H} t} \phi + w(x,t)),
\end{eqnarray*}
where $e^{i \mathcal{H} t} \phi + w(x,t)$ disperses as $t \to \infty$.  Hence, we have long time persistance of a soliton of minimal mass despite the fact that these solutions are shown to be nonlinearly unstable in \cite{CP1}.

\end{abstract}

\maketitle    

\section{Introduction}

In this result, we develop the dipsersive estimates used to prove stability of solitons for a focusing, saturated nonlinear Schr\"odinger equation (NLS) in $\R \times \R^d$:
\begin{eqnarray*}
i u_t + \Delta u + \beta (|u|^2) u & = & 0 \\
u(0,x) & = & u_0 (x),
\end{eqnarray*}
where $\beta: \reals \to \reals$, $\beta(s) \geq 0$ for all $s \in \reals$, $\beta$ has a specific structure outlined in the following definitions:

\begin{defn}
\label{def:non1}
Saturated nonlinearities of type $1$ are of the form
\begin{eqnarray}
\label{sat:eqn1}
\beta (s) = s^{\frac{q}{2}} \frac{s^{\frac{p-q}{2}}}{1 + s^{\frac{p-q}{2}}}, 
\end{eqnarray}
where $p > 2+\frac{4}{d}$ and $\frac{4}{d}  > q > 0$ for $d \geq 3$ and $\infty > p > 2+\frac{4}{d} > \frac{4}{d} > q > 0$ for $d < 3$.
\end{defn}    

\begin{defn}
\label{def:non2}
Saturated nonlinearities of type $2$ are of the form
\begin{eqnarray}
\label{sat:eqn2}
\beta(s) = \frac{s}{(1 + s)^{\frac{2-q}{2}}},
\end{eqnarray}
where $\frac{4}{d} > q >0$, $d > 2$. 
\end{defn}

\begin{rem}
In both cases, for $|u|$ large, the behavior is $L^2$ subcritical and for $|u|$ small, the behavior is $L^2$ supercritical.  For Definition \ref{def:non1}, $p$ is chosen much larger than the $L^2$ critical exponent, $\frac{4}{d}$ in order to allow sufficient regularity when linearizing the equation.  
\end{rem}

In the sequel, we assume that $u_0 \in H^1$ and $|x|u_0 \in L^2$, or in other words, $u_0$ has finite variance.  For this initial data, from the spatial and phase invariance of NLS, we have many the following conserved quantities:

Conservation of Mass (or Charge):
\begin{eqnarray*}
Q (u) = \frac{1}{2} \int_{\R^n} |u|^2 dx = \frac{1}{2} \int_{\R^d} |u_0|^2 dx,
\end{eqnarray*}

and

Conservation of Energy:
\begin{eqnarray*}
E(u) = \int_{\R^d} | \nabla u |^2 dx - \int_{\R^d} G(|u|^2) dx = \int_{\R^d} | \nabla u_0 |^2 dx - \int_{\R^d} G(|u_0|^2) dx,
\end{eqnarray*}
where 
\begin{eqnarray*}
G(t) = \int_0^t \beta(s) ds.
\end{eqnarray*}

We also have the pseudoconformal conservation law:
\begin{eqnarray}
\| (x + 2 i t \nabla ) u \|^2_{L^2} - 4 t^2 \int_{\R^d} G(|u|^2) dx = \| x \phi \|^2_{L^2} - \int_0^t \theta (s) ds,
\end{eqnarray}
where 
\begin{eqnarray*}
\theta (s) = \int_{\R^d} (4 (d+2) G(|u|^2) - 4 d \beta(|u|^2) |u|^2) dx.
\end{eqnarray*}
Note that $(x + 2 i t \nabla )$ is the Hamilton flow of the linear Schr\"odinger equation, so the above identity relates how the solution to the nonlinear equation is effected by the linear flow.

Detailed proofs of these conservation laws can be arrived at easily using energy estimates or Noether's Theorem, which relates conservation laws to symmetries of an equation.  Global well-posedness in $L^2$ of (NLS) with $\beta$ of type $1$ or $2$ for finite variance initial data follows from standard theory for $L^2$ subcritical monomial nonlinearities.  Proofs of the above results can be found in numerous excellent references for (NLS), including \cite{Caz} and \cite{SS}.

{\sc Acknowledgments.} This paper is a result of a thesis done under the direction of Daniel Tataru at the University of California, Berkeley.  It is fair to say this work would not exist without his assistance.  The work was supported by Graduate Fellowships from the University of California, Berkeley and NSF grants DMS0354539 and DMS0301122.  In addition, the author spent a semester as an Associate Member of MSRI during the development of these results.  Currently, the author is supported by an NSF Postdoctoral Fellowship.

\section{Soliton Solutions}

A soliton solution is of the form 
\begin{eqnarray*}
u(t,x) = e^{i \lambda t} R_\lambda(x)
\end{eqnarray*} where 
$\lambda > 0$ and $R_\lambda (x)$ is a positive, radially symmetric, exponentially decaying solution of the equation:
\begin{eqnarray}
\label{eqn:sol}
\Delta R_\lambda - \lambda R_\lambda + \beta (R_\lambda ) R_\lambda = 0.
\end{eqnarray}
With this type of nonlinearity, soliton solutions exist and are known to be unique.  Existence of solitary waves for nonlinearities of the type presented in Definitions \ref{def:non1} and \ref{def:non2} is proved by in \cite{BeLi} by minimizing the functional 
$$T(u) = \int | \nabla u |^2 dx$$ 
with respect to the functional 
$$V(u) = \int [ G(|u|^2) - \frac{\lambda}{2} |u|^2 ] dx.$$  
Then, using a minimizing sequence and Schwarz symmetrization, one sees the existence of the nonnegative, spherically symmetric, decreasing soliton solution.  For uniqueness, see \cite{Mc}, where a shooting method is implemented to show that the desired soliton behavior only occurs for one particular initial value. 

An important fact is that $Q_{\lambda} = Q(R_{\lambda})$ and $E_{\lambda} = E(R_{\lambda})$ are differentiable with respect to $\lambda$.  This fact can be determined from the early works of Shatah, namely \cite{Sh1}, \cite{Sh2}.  By differentiating Equation \eqref{eqn:sol}, $Q$ and $E$ with respect to $\lambda$, we have
\begin{eqnarray*}
\partial_{\lambda} E_{\lambda} = - \lambda \partial_{\lambda} Q_{\lambda}.
\end{eqnarray*}

Numerics show that if we plot $Q_{\lambda}$ with respect to $\lambda$, we get a curve that goes to $\infty$ as $\lambda \to 0, \infty$ and has a global minimum at some $\lambda = \lambda_0 > 0$, see Figure \ref{f:solcurves}.  We will explore this in detail in a subsequent numerical work \cite{Mnum}.  Variational techniques developed in \cite{GSS} and \cite{ShSt} tell us that when $\delta ( \lambda ) = E_{\lambda} + \lambda Q_{\lambda}$ is convex, or $\delta '' (\lambda) > 0$, we are guaranteed stability under small perturbations, while for $\delta '' (\lambda) < 0$ we are guaranteed that the soliton is unstable under small perturbations.  For brief reference on this subject, see \cite{SS}, Chapter 4.  For nonlinear instability at a minimum, see \cite{CP1}.  For notational purposes, we refer to a minimal mass soliton as $R_{min}$.

\begin{figure}
\scalebox{0.38}{\includegraphics{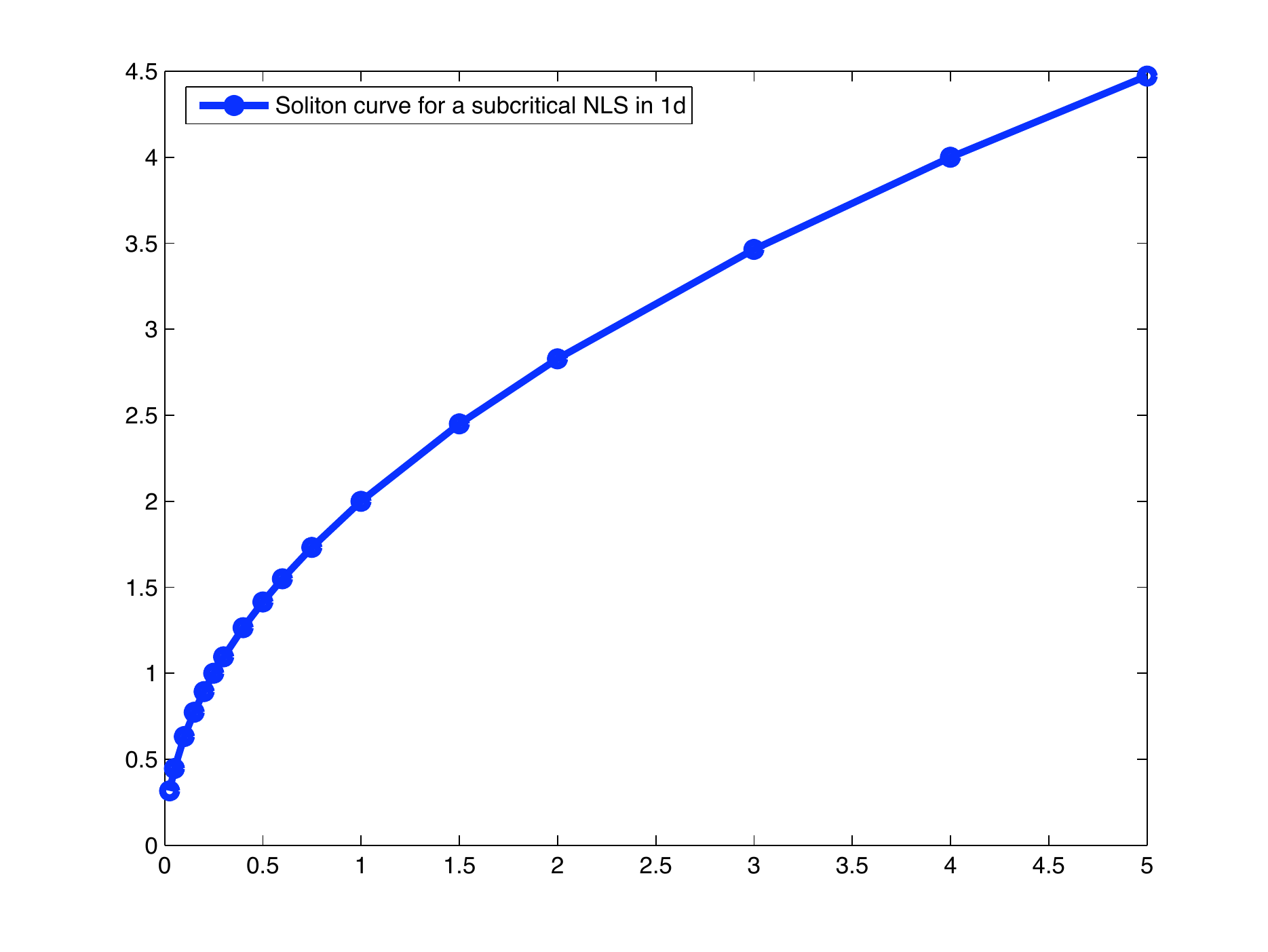}} \hfill \scalebox{0.38}{\includegraphics{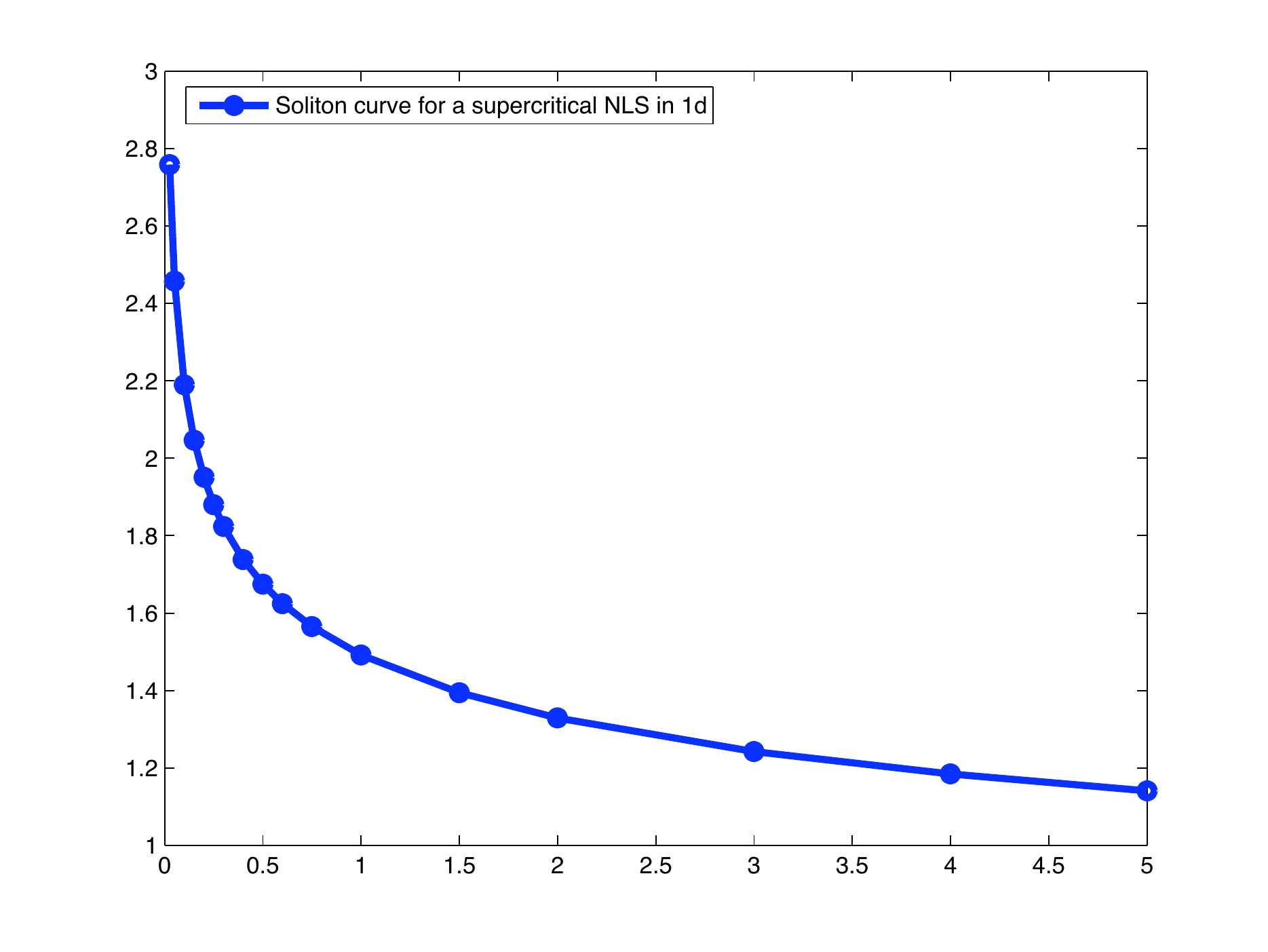}} 
\scalebox{0.38}{\includegraphics{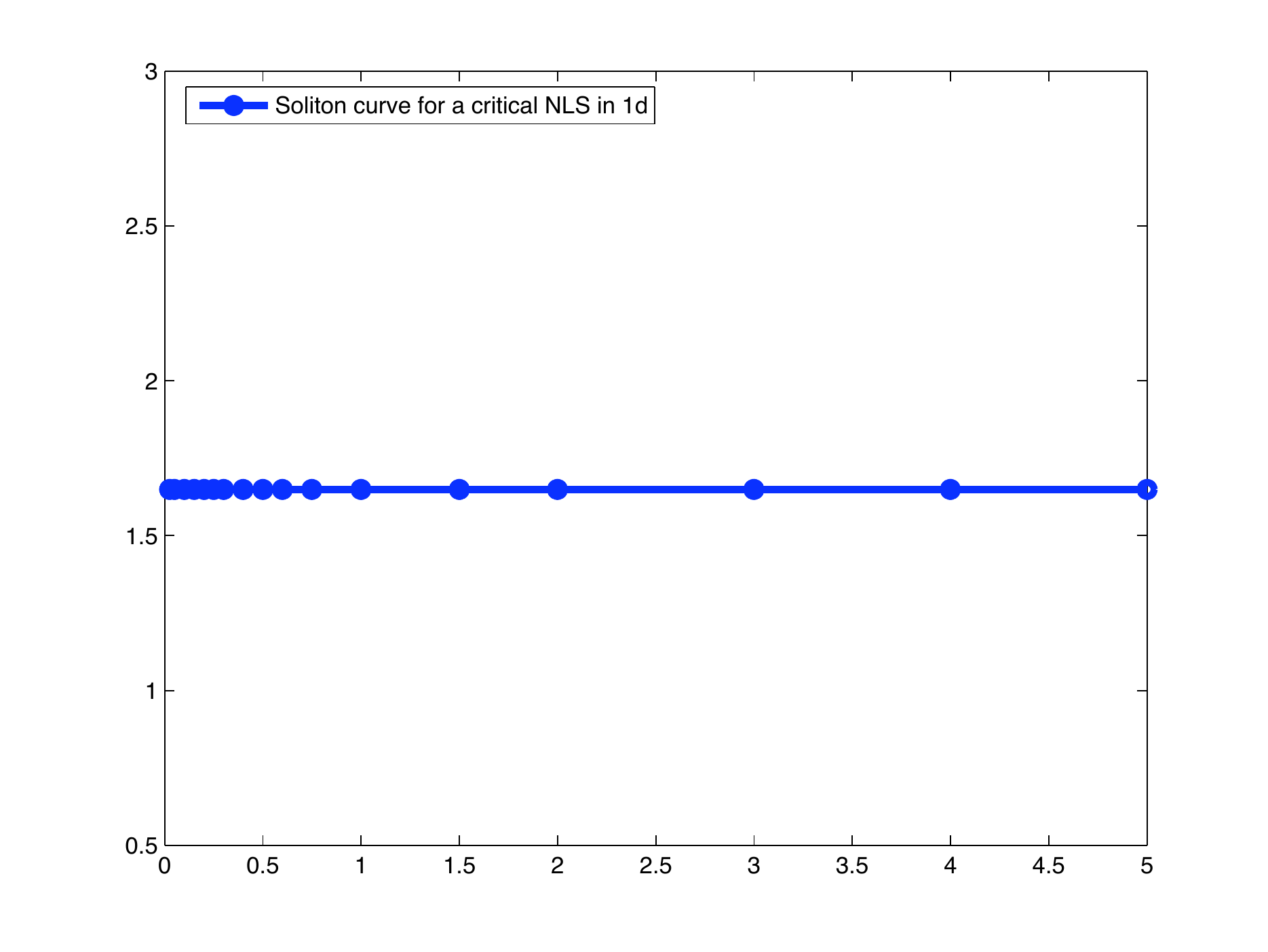}} \hfill \scalebox{0.38}{\includegraphics{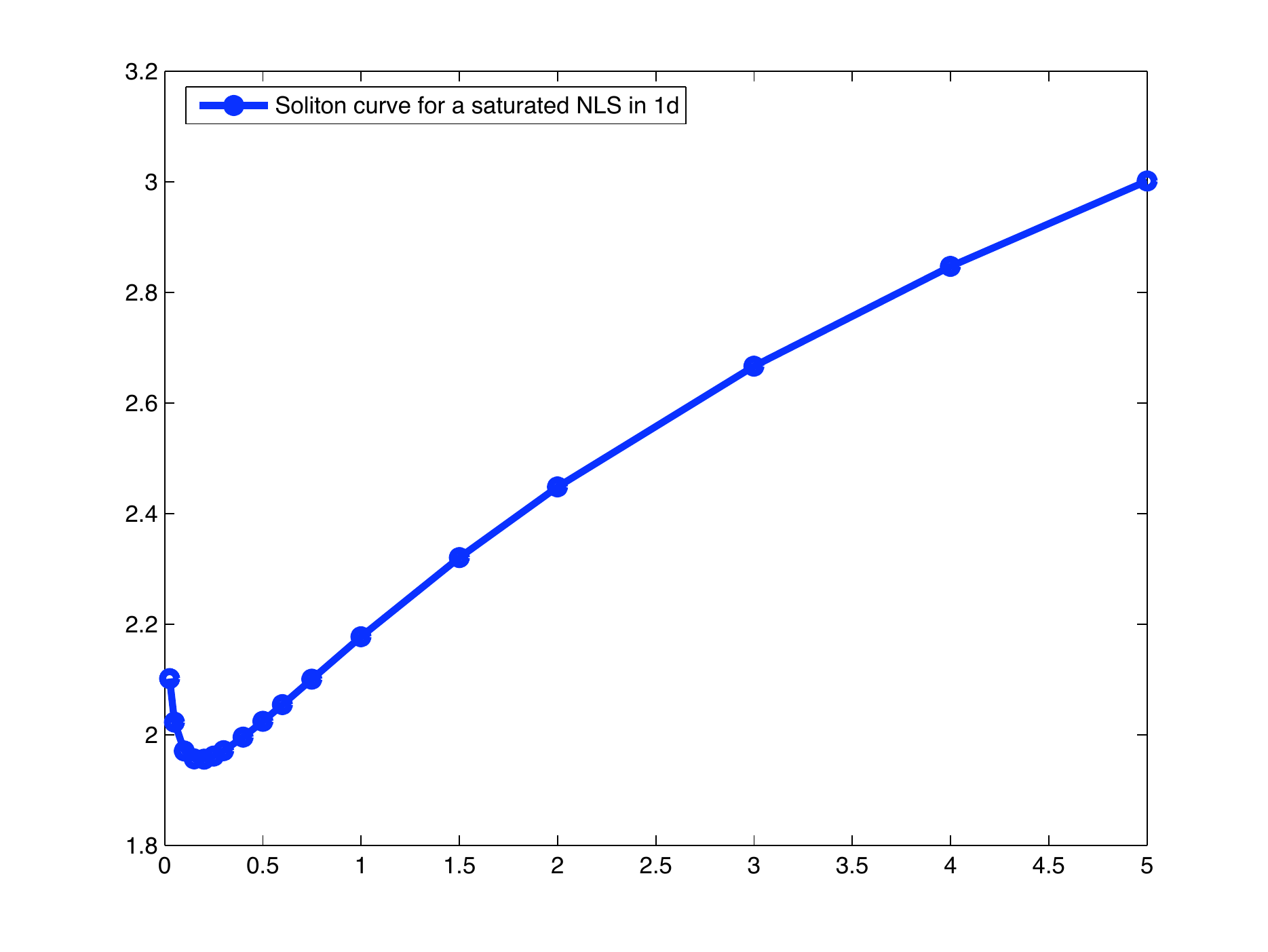}} 
\scalebox{0.38}{\includegraphics{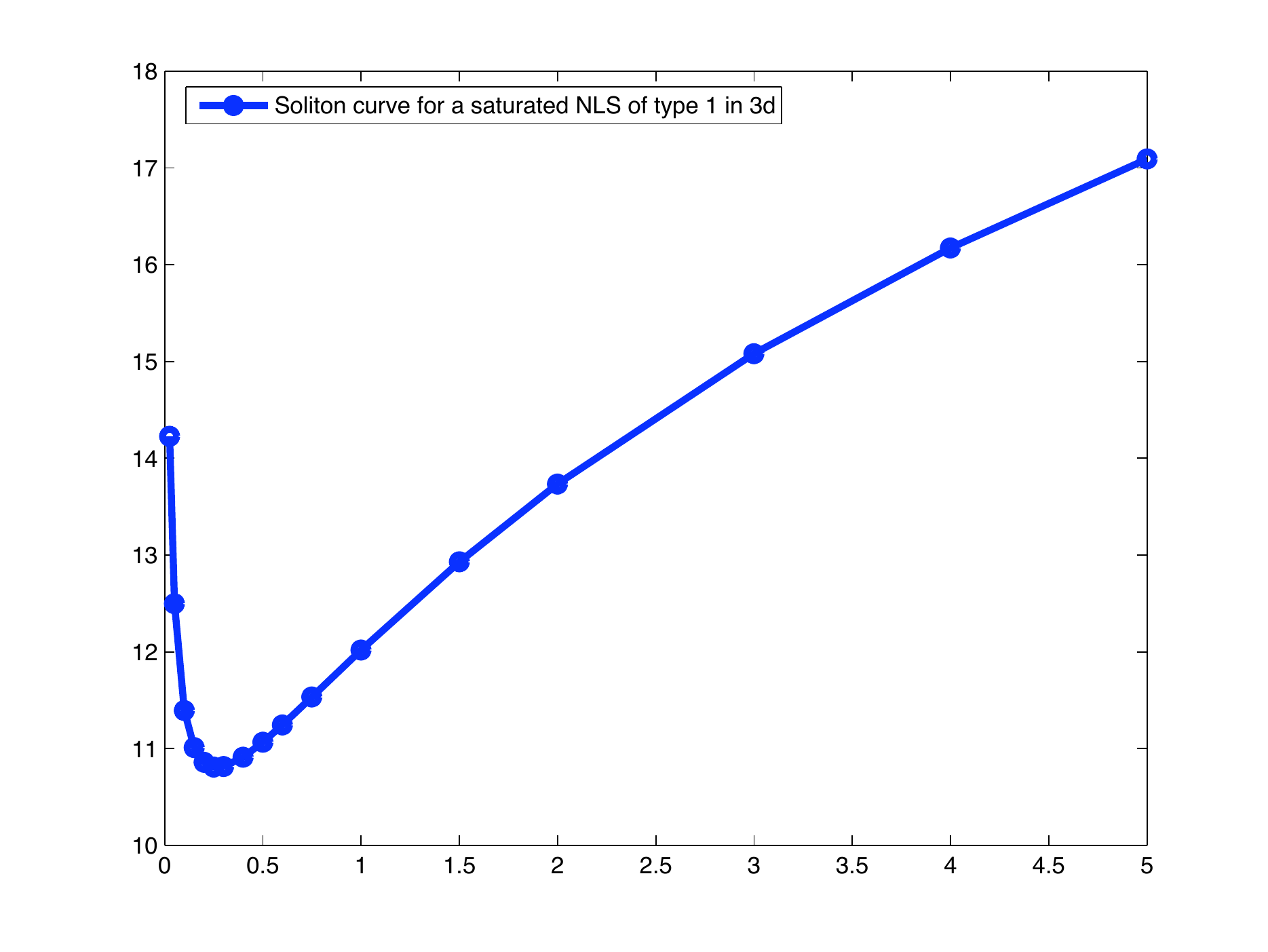}} \hfill \scalebox{0.38}{\includegraphics{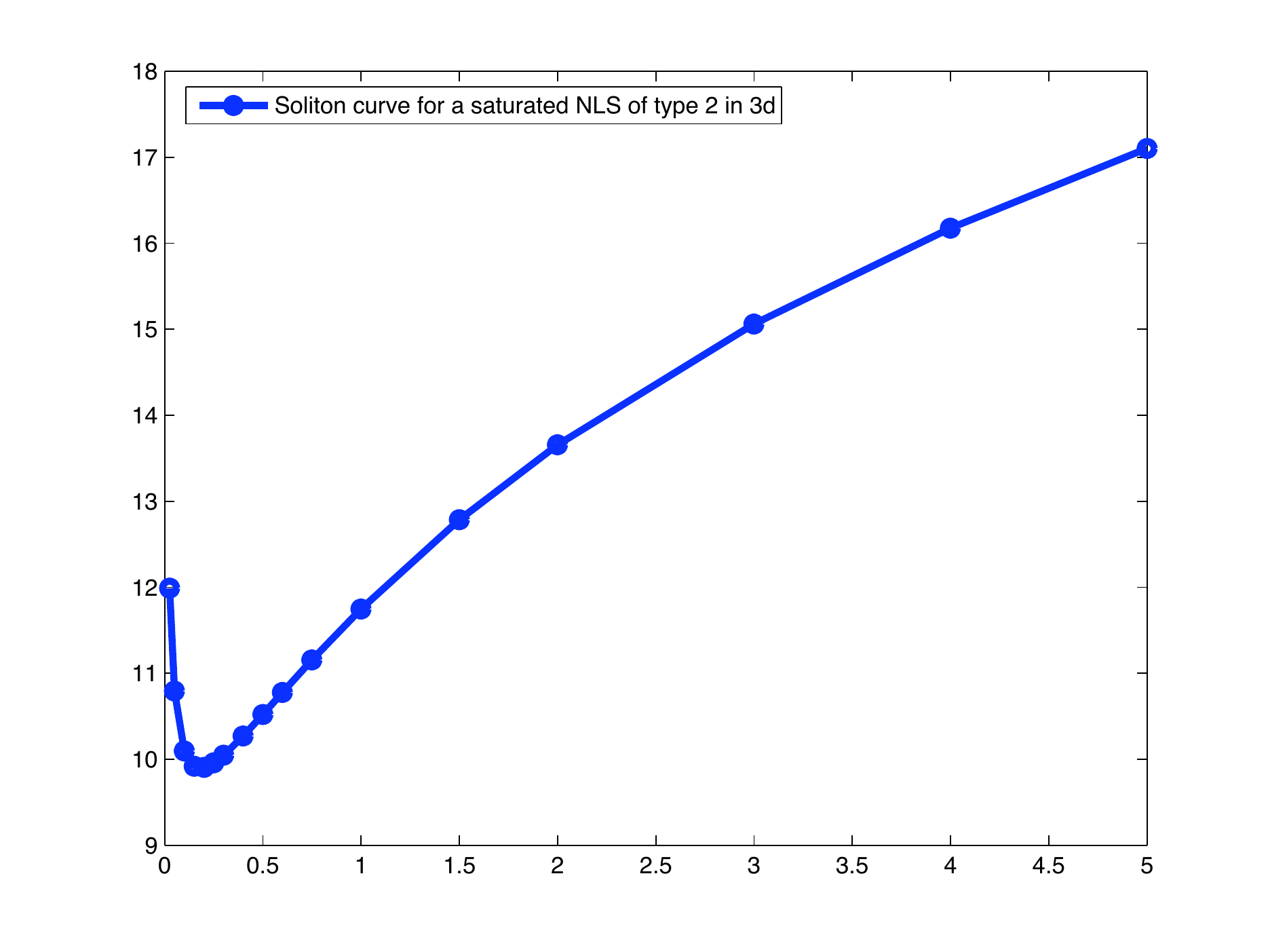}}

\caption{Plots of the soliton curves ($Q(\lambda)$ with respect to $\lambda$) for a subcritical nonlinearity ($d=1$, $p=3$), supercritical nonlinearity ($d=3$, $p=3$), critical nonlinearity ($d=1$, $p=5$), saturated nonlinearity of type $1$ ( $p=7$, $q=3$) in $\reals$, saturated nonlinearity of type $1$ in $3d$ ($p=4$, $q=2$), saturated nonlinearity of type $2$ in $\reals^3$ ($q=2$).  The curves for the monomial nonlinearities are found analytically, while the curves for the saturated nonlinearities are found numerically.}
\label{f:solcurves}
\end{figure}

\section{Linearization about a Soliton}

Let us write down the form of NLS linearized about a soliton solution.  First of all, we assume we have a solution $\psi = e^{i \lambda t}(R_{\lambda} + \phi(x,t))$.  For simplicity, set $R = R_\lambda$.  Inserting this into the equation we know that since $\phi$ is a soliton solution we have 
\begin{eqnarray}
i (\phi)_t + \Delta (\phi) & = & -\beta( R^2) \phi - 2 \beta'( R^2 )  R^2 \text{Re}(\phi) + O (\phi^2),
\end{eqnarray}
by splitting $\phi$ up into its real and imaginary parts, then doing a Taylor Expansion. Hence, if $\phi = u + iv$, we get
\begin{eqnarray}
\partial_t \left( \begin{array}{c}
u \\
v
\end{array} \right) = \mathcal{H} \left( \begin{array}{c}
u \\
v
\end{array} \right),
\end{eqnarray}
where 
\begin{eqnarray}
\mathcal{H} =  \left( \begin{array}{cc}
0 & L_{-} \\
-L_{+} & 0
\end{array} \right),
\end{eqnarray}
where $$L_{-} = - \Delta + \lambda - \beta( R_\lambda )$$ and 
$$L_{+} = - \Delta + \lambda - \beta( R_\lambda ) - 2 \beta' (R^2_\lambda) R_\lambda^2.$$

\begin{defn}
\label{spec:defn1}
A Hamiltonian, $\mathcal{H}$ is called admissible if the following hold: \\
1)  There are no embedded eigenvalues in the essential spectrum, \\
2)  The only real eigenvalue in $[-\lambda, \lambda ]$ is $0$, \\
3)  The values $\pm \lambda$ are not resonances. \\
\end{defn}

\begin{defn}
\label{spec:defn2}
Let (NLS) be taken with nonlinearity $\beta$.  We call $\beta$ admissible if there exists a minimal mass soliton, $R_{min}$, for (NLS) and the Hamiltonian, $\mathcal{H}$, resulting from linearization about $R_{min}$ is admissible in terms of Definition \ref{spec:defn1}.
\end{defn}

The spectral properties we need for the linearized Hamiltonian equation in order to prove stability results are precisely those from Definition \ref{spec:defn1}.  Notationally, we refer to $P_d$ and $P_c$ as the projections onto the discrete spectrum of $\mathcal{H}$ and onto the continuous spectrum of $\mathcal{H}$ respectively.

Analysis of these spectral conditions will be done both numerically and analytically in the forthcoming work \cite{Mspec}.

\section{Review of Dispersive Estimates}

We review here the disersive estimates from \cite{Mlin}.  Let $\mathcal{S}$ be the Schwartz class of functions.  Then, we have the following results:  

\begin{thm}
\label{thm:tdec}
Given an admissible Hamiltonian $\mathcal{H}$, $P_c$ the projection on the continuous spectrum of $\mathcal{H}$, for initial data $\phi \in \mathcal{S}$, we have
\begin{eqnarray*}
\| e^{it\mathcal{H}} P_c \phi \|_{L^\infty} \leq t^{-\frac{d}{2}}.
\end{eqnarray*}
\end{thm}

\begin{thm}
\label{thm:wlindec}
Let $\mathcal{H}$ be an admissible Hamiltonian as defined above.  Let $\tilde{\phi}_{\xi}$ be the associated distorted Fourier basis.  Assume $\vec{\psi} \in L^{1,M}$ and
\begin{eqnarray}
\label{eqn:mom1}
\partial^\alpha_\xi \partial^\beta_{|\xi|} \vec{\Psi} (0) = 0,
\end{eqnarray}
for multi-indices $\alpha$, $\beta$ such that $|\alpha| + |\beta| = 0,1,2,\dots,2M$, where
\begin{eqnarray*}
\vec{\Psi} (\xi) = \int_y \tilde{\phi}_{\xi} (y) \vec{\psi} (y) dy.
\end{eqnarray*}  
Then,
\begin{eqnarray}
\label{eqn:lindecw}
\| e^{-c |x|} e^{i t \mathcal{H}} P_{c} \vec{\psi} \|_{L^\infty} \leq C t^{-\frac{d}{2}-M} \| \vec{\psi} \|_{L^{1,M}},
\end{eqnarray}
for any $c > 0$.
\end{thm}

From Theorems \ref{thm:tdec} and \ref{thm:wlindec}, we have the following results:

\begin{thm}[Erdogan-Schlag,Bourgain]
\label{thm:disp1}

Let $P_c$ and $P_d$ be projections onto the continuous and discrete spectrum of $\mathcal{H}$ respectively.  Then,
\begin{eqnarray*}
(i) \ \| e^{it \mathcal{H}} P_c \phi \|_{H^1} & \leq & C \| \phi \|_{H^1} \\
(ii) \ \| e^{it \mathcal{H}} (P_c\phi) \|_{H^s} & \leq & C \| \phi \|_{H^s} \\
(iii) \ \| e^{it \mathcal{H}} (P_d\phi) \|_{H^s} & \leq & C (1 + |t|^3) \int e^{-c|x|} |\phi(x)| dx \\
(iv) \ \| |x|^\alpha e^{it \mathcal{H}} (P_c \phi) \|_{L^2} & \leq & C( \| |x|^\alpha \phi \|_{L^2} + (1+|t|^\alpha) \| \phi \|_{H^\alpha}) \\
(v) \ \| |x|^\alpha e^{it \mathcal{H}} (P_d \phi) \|_{L^2} & \leq & C (1+|t|^3) \int |\phi| e^{-c|x|} dx.
\end{eqnarray*}
\end{thm}

\begin{thm}
\label{thm:strich1}
For $p$ and $p'$ such that $\frac{1}{p} + \frac{1}{p'} = 1$, with $2 \leq p \leq \infty$, and $t \neq 0$, the transformation $e^{i \mathcal{H} t}$ maps continuously $L^{p'} ( \R^d)$ into $L^{p} ( \R^d)$ and
\begin{eqnarray}
\label{eqn:strich1}
\| e^{i \mathcal{H} t} \phi \|_{L^p} \lesssim \frac{1}{|t|^{d(\frac{1}{2}-\frac{1}{p})}} \| \phi \|_{L^{p'}}.
\end{eqnarray}
\end{thm}

\begin{thm}[Schlag]
\label{thm:strich2}
For every $\phi \in L^2$ and every admissible pair $(q,r)$, the function
$t \to e^{i \mathcal{H} t} \phi$ belongs to $L^q (\reals, L^r (\reals^d)) \cap C ( \reals, L^2 (\reals^d))$, and there exists a constant $C$ depending only on $q$ such that
\begin{eqnarray}
\label{eqn:strich2}
\| e^{i \mathcal{H} t} \phi \|_{L^q (\reals, L^r (\reals^d))} \leq C \| \phi \|_{L^2}.
\end{eqnarray}
\end{thm}

It should be noted that similar estimates to those in Theorems \ref{thm:tdec} and \ref{thm:wlindec} were proven in the works \cite{ES1} and \cite{BW}, where in the first the techniques used were more along the lines of resolvent estimates and in the second the fact that the nonlinearities of interest were of even integer powers was crucial to the argument.  However, from the scattering theory point of view taken in \cite{Mlin}, we are able to define the orthogonality condition \eqref{eqn:mom1} in order to generalize the results and more easily prove the weighted estimates.

\section{Main Results}
\label{int:main}

To begin, we define the function space
\begin{eqnarray*}
\mathcal{P}^A_1 = \{ \phi \in L^2 | \| \phi \|_{H^A} < \infty, \ \| |x|^A \phi \|_{L^2} < \infty,  \ \int x^\alpha \phi(x) dx = 0 \ \text{for $|\alpha| \leq 2 A$} \},
\end{eqnarray*}
with norm given by
\begin{eqnarray*}
\| \phi \|_{\mathcal{P}^A_1} = \left( \| \phi \|_{H^A}^2 + \| |x|^A \phi \|_{L^2}^2 \right)^{\frac{1}{2}}.
\end{eqnarray*} 

We similarly define the function space
\begin{eqnarray*}
\mathcal{P}^A_2 = \{ \phi \in P_c \mathcal{H} | \| \phi \|_{H^A} < \infty, \ \| |x|^A \phi \|_{L^2} < \infty,  \ \text{condition \eqref{eqn:mom1} is satisfied for $j \leq A$} \},
\end{eqnarray*}
with norm given by
\begin{eqnarray*}
\| \phi \|_{\mathcal{P}^A_2} = \left( \| \phi \|_{H^A}^2 + \| |x|^A \phi \|_{L^2}^2 \right)^{\frac{1}{2}}.
\end{eqnarray*}

In this result, we seek to prove that minimal mass solitons for nonlinear Schr\"odinger equations in three dimensions have stable perturbations for long times.  These minimal solitons are unstable as discussed below. The main goal of this thesis is to prove the following three theorems:

\begin{thm}
\label{thm:mainLS}
Take the equation in $\RR \times \RR^3$
\begin{eqnarray}
\label{eqn:NLS}
\left\{\begin{array}{c}
i u_t + \Delta u + \beta (|u|^2) u = 0 \\
u(0,x) = u_0 (x),
\end{array} \right.
\end{eqnarray}
where $\beta$ is an admissible saturated nonlinearity of type $1$.  For any $\phi \in \mathcal{P}^A_1$, Equation \eqref{eqn:NLS} has a solution $u$ for $ t \in [\frac{1}{\delta}, \infty)$ of the form 
\begin{eqnarray*}
u(x,t) = R_{min} + v(t) = R_{min} + e^{i \Delta t} \phi + w(x,t),
\end{eqnarray*}
where $R_{min}$ is the minimal mass soliton in Definition \ref{spec:defn1} and $\| w(\cdot,t) \|_{H^2}  \to 0$ as $t \to \infty$.  For Equation \eqref{eqn:NLS} in $\reals^3$ of type $1$, we have $A > \frac{13}{2}$.  
\end{thm}

\begin{thm}
\label{thm:mainHLS}
Take Equation \eqref{eqn:NLS}, where $\beta$ is an admissible saturated nonlinearity of type $1$.  For any $\phi \in \mathcal{P}^A_1$, Equation \eqref{eqn:NLS} has a solution $u$ for $t \in [\frac{1}{\delta}, \infty)$ of the form 
\begin{eqnarray*}
u(x,t) = R_{min} + v(t) = R_{min} + e^{i \mathcal{H} t} \phi + w(x,t),
\end{eqnarray*}
where $R_{min}$ is the minimal mass soliton in Definition \ref{spec:defn1} and $\| w(\cdot,t) \|_{L^2}  \to 0$ as $t \to \infty$.  In this theorem, for Equation \eqref{eqn:NLS} in $\reals^3$ of type $1$, we have $A > \frac{5}{2}$.  
\end{thm}

\begin{thm}
\label{thm:longtime}
Given Equation \eqref{eqn:NLS}, where $\beta$ is an admissible saturated nonlinearity of type $2$, for any $\phi = P_c \phi \in W^{2,1} \cap H^2$ with $\| \phi \|_{W^{2,1} \cap H^2} < \delta < 1$, Equation \eqref{eqn:NLS} has a solution for $t \in [ 0,\left( \frac{1}{2 \delta} \right)^{\frac{1}{4}} )$ of the form 
\begin{eqnarray*}
u(x,t) = R_{min} + v(t) = R_{min} + e^{i \mathcal{H} t} \phi + w(x,t),
\end{eqnarray*}
where $R_{min}$ is the minimal mass soliton in Definition \ref{spec:defn1} and
\begin{eqnarray*}
u(x,0) = R_{min} + \phi.
\end{eqnarray*} 
\end{thm}

\begin{rem}
In Theorems \ref{thm:mainLS} and \ref{thm:mainHLS}, the stable perturbations can be shown to live on a finite codimension manifold for $p$ large enough compared to $d$.  This will be explored further below.
\end{rem}

The class of functions $\mathcal{P}^A_i$ for $i = 1,2$ will be developed throughout the course of this work.  They will result from projecting onto a distorted Fourier basis for the linearized problem.  For a further discussion these topics and the notion of distorted Fourier basis, see \cite{Mlin} and the refences contained within.

\section{Preliminaries}
\label{thm1:prelim}

We wish to construct a contraction argument similar to that presented in \cite{BW} in the case where we have a more general nonlinearity.  In particular, we have the equation
\begin{eqnarray*}
\left\{ \begin{array}{c}
i u_t + \Delta u+ F(|u|^2) u = 0, \\
u(0,x) = u_0 (x),
\end{array} \right.
\end{eqnarray*}
where $F$ is chosen to be of type $1$ or type $2$.

Many estimates that hold for the $L^2$ critical equation hold at that soliton because they share the property that $\partial_\lambda Q(u_\lambda) = 0$ where $\lambda$ is the soliton parameter and $Q$ is the $L^2$ mass.  As this is a minimal mass soliton, there are many possible perturbations.  One could perturb onto the manifold of stable solitons, onto the manifold of unstable solitons, or in fact, reduce the $L^2$ energy so that solitons no longer formed.  Unfortunately, due to a lack of scaling and general difficulties, very little is known about stable perturbations to such a soliton.  Also, it is a major question whether or not we have dispersion and scattering for initial data with $L^2$ mass below the minimal soliton mass.  We hope to address this in future work, but for now we wish to prove the existence of stable solutions to the minimal mass soliton.  We may assume that the minimal mass soliton occurs at $\lambda_0 = 1$.  In other words, if $R$ is the desired soliton, we seek a solution of the form
\begin{eqnarray*}
u = R e^{it} + z_{\phi} e^{it} + w e^{it},
\end{eqnarray*}
where $w \in C( [\frac{1}{\delta}, \infty]; X)$ and $\| w \|_X \leq \frac{1}{t^N}$ for some normed space $X$ and some large $N$ to be determined.  The goal is to solve this problem for $z$ solving both
\begin{eqnarray*}
i z_t + \Delta z & = &0, \\
z(0,x) & = & \phi(x),
\end{eqnarray*}
as well as
\begin{eqnarray*}
i z_t + \mathcal{H} z & = & 0, \\
z(0,x) & = & \phi(x),
\end{eqnarray*}
for $\mathcal{H}$ the matrix Hamiltonian that results from linearizing about the minimal mass soliton.  

To begin, we run through the contraction argument assuming that we are using the linear Schrodinger operator, $e^{i \Delta t}$, and the space $X = X_A$ defined by
\begin{eqnarray*}
X_A = \{ \phi | \| \phi \|_{H^A} + \| (1 + |x|)^A \phi \|_{L^2} < \infty \}.
\end{eqnarray*}

Let $v_0 = z_\phi e^{-it}$ and let $u(x,t) = e^{it} (R + v)$ for $v = w+v_0$.  Then, $v$ must satisfy
\begin{eqnarray*}
iv_t + \Delta v - v + [F(|R+v|^2) (R+v) - F(R^2) R] = 0,
\end{eqnarray*}
or
\begin{eqnarray*}
i v_t + \Delta v - v + (F(R^2) + F'(R^2) R^2) v + (F'(R^2) R^2) \bar{v}  + O(|v|^2) = 0.
\end{eqnarray*}

Since $i (v_0)_t + \Delta (v_0) - v_0 = 0$, we have
\begin{eqnarray*}
i w_t + \Delta w - w + [F(|R+v_0+w|^2) (R+v_0+w) - F(R^2) R] = 0.
\end{eqnarray*}

Let 
\begin{eqnarray*}
f_0 & = & F(|R + v_0|^2)(R+v_0) - F(R^2) R, \\
a & = & [F(|R+v_0|^2) + F'(|R+v_0|^2)|R+v_0|^2] - [F(R^2) + F'(R^2)R^2],  \\
b & = & F'(|R+v_0|^2)(R+v_0)^2 - F'(R^2) R^2, \\
G(w) & = & F(|R+v_0+w|^2) (R+v_0+w) - F(|R + v_0|^2)(R+v_0) \\
& - & [F(|R+v_0|^2) + F'(|R+v_0|^2)|R+v_0|^2] w - F'(|R+v_0|^2)(R+v_0)^2 \bar{w}.
\end{eqnarray*}
Then, we have
\begin{eqnarray*}
i w_t + \Delta w - w + (F(R^2) + F'(R^2)R^2) w + F'(R^2)R^2 \bar{w} + f_0 + a w + b \bar{w} + G(w) = 0.
\end{eqnarray*}

In other words, we have
\begin{eqnarray*}
i w_t - H w + a w + b \bar{w} + f_0 + G(w) = 0,
\end{eqnarray*}
where $G$ is at least quadratic in $w$ and $f_0$ is linear in $v_0$.  

To see this, note that for nonlinearities of type $1$, we have
\begin{eqnarray*}
F(x) & = & \frac{x^{\frac{p}{2}}}{1+x^{\frac{p-q}{2}}} \\
F'(x) & = & \frac{x^{\frac{p}{2}-1} (\frac{p}{2}+ \frac{q}{2} x^{\frac{p-q}{2}})}{(1+x^{\frac{p-q}{2}})^2} \\
F''(x) & = & \frac{x^{\frac{p}{2}-2} (\frac{p}{2}(\frac{p}{2}-1)+(pq-\frac{q^2}{4}-\frac{q}{2}-\frac{p^2}{2}) x^{\frac{p-q}{2}}+(\frac{q^2}{4}-\frac{q}{2}) x^{p-q})}{(1+x^{\frac{p-q}{2}})^3},
\end{eqnarray*}
and for type $2$,
\begin{eqnarray*}
F(x) & = & \frac{x}{(1+x)^{\frac{2-q}{2}}} \\
F'(x) & = & \frac{ 1 + \frac{q}{2} x}{(1+x)^{2 - \frac{q}{2}}} \\
F''(x) & = & \frac{(q-2)+\left( \frac{q^2}{4}-\frac{q}{2} \right)x}{(1+ x)^{3 - \frac{q}{2}}}.
\end{eqnarray*}

Note that in both cases, $F \in C^1$ and in the second case, $F \in C^\infty$.  However, we can define $G(z,\bar{z}) = F(|R+z|^2)(R+z)$.  This is $C^2$ at $z = 0$ in both cases.  To see this, note
\begin{eqnarray*}
\partial_z G & = & F'(|R+z|^2) (R+z)(R+\bar{z}) + F(|R+z|^2) \\
\partial_{\bar{z}} G & = & F'(|R+z|^2) (R+z)^2 \\
\partial_{zz} G & = & 2 F'(|R+z|^2)(R+\bar{z}) + F''(|R+z|^2) (R+z)(R+\bar{z})^2 \\
\partial_{\bar{z} \bar{z}} G & = & F''(|R+z|^2) (R+z)^3 \\
\partial_{z \bar{z}} G & = & 2 F'(|R+z|^2)(R+z) + F''(|R+z|^2) (R+z)^2(R+\bar{z}),
\end{eqnarray*}
hence at $z = 0$, the terms resulting in exponential growth from $F''$ are controlled.  In the resulting Taylor expansion, we see 
\begin{eqnarray*}
G(z,z') = F(R^2)R + F'(R^2)R^2 \bar{z} + (F'(R^2)R^2 + F(R^2)) z + O(|R+z|^{p-1} |z|^2).
\end{eqnarray*}

Let us make the assumption that we are working with type $1$ nonlinearities with $\frac{10}{3}<p<\infty$.  The author believes that similar results should hold even if $p$ is not restricted to allow more regularity of the nonlinearity, however for the expansions in the sequel to be accurate, we must restrict the nonlinearities to have sufficient regularity, as well as to provide sufficient decay in $t$.

Let us explore the behaviors of the the above functions.  For simplicity, let $v_0 = v_1+i v_2$.  To begin,
\begin{eqnarray*}
|f_0| & = & |F(R^2+2Rv_1+v_1^2+v_2^2)(R+v_0)-F(R^2)R| \\
& \lesssim & |[F(R^2)R + O(|R+v_0|^p|v_0|) - F(R^2)R| \\
& \lesssim & O(R^p |v_0|) + O(|v_0|^{p+1})| .
\end{eqnarray*}

A similar calculation gives that 
\begin{eqnarray*}
|a| \lesssim O(R^{p-1} |v_0|) + O(|v_0|^{p}),
\end{eqnarray*}
and similarly for $b$, we have
\begin{eqnarray*}
|b| \lesssim O(R^{p-1} |v_0|) + O(|v_0|^{p}).
\end{eqnarray*}

Finally, we have
\begin{eqnarray*}
|G| \lesssim O( |w|^2 ).
\end{eqnarray*} 

Hence, we solve the integral equation for $w$
\begin{eqnarray*}
w(t) = -i \int_t^\infty e^{i(\tau-t)H} [f_0 + aw + b \bar{w} + G(w)] d\tau.
\end{eqnarray*}

We would like to see that $\| w(t) \|_{X_A} \leq \frac{1}{t^{M}}$ for some $M$ to be determined in order to show that we have a stable manifold of perturbations on the function space $X^A$.  However, we are actually only able to prove $\| w(t) \|_{X^0} \leq \frac{1}{t^M}$, where $X^A \subset X^0$.  The resulting effects of this will appear later in Section \ref{thm1:man}.

\section{Contraction Argument}
\label{thm1:cont}

Making the assumption that 
\begin{eqnarray*}
\int x^\beta \phi(x) dx = 0,
\end{eqnarray*}
for the multi-index $|\beta| = 0,1,2,...,2M$ for $M$ to be determined, we have by Taylor expanding the exponential in the fundamental solution that
\begin{eqnarray*}
z_\phi(x,t) = O \left( \frac{1}{t^{M+\frac{d}{2}}} \right)
\end{eqnarray*}
where $|x| \lesssim 1$.

We have that 
\begin{eqnarray*}
\| D^\alpha v_0 \|_{L^\infty} & \leq & \frac{C}{t^{\frac{d}{2} p}}, \\
| e^{-c|x|} D^\alpha v_0 (x,t)| & < & \frac{C}{t^N},
\end{eqnarray*}
for $t$ and $N$ large and the range of $\alpha$'s to be determined.  In general, we explore low regularity perturbations but are in need of $L^\infty$ bounds, so $\alpha$ will be small, but positive.

Then, we have for $s$ in some range to be determined
\begin{eqnarray*}
\| (R v_0) (t) \|_{H^s} \leq \frac{C}{t^{M+\frac{d}{2}}},
\end{eqnarray*}
and hence, we have to check that for our space $X^A$, 
\begin{eqnarray*}
\| f_0 \|_{X^A} & \leq & \frac{C}{t^{\frac{d}{2} (p+1)}}.
\end{eqnarray*}
where $p$ is determined by the supercritical power in the nonlinearity and we have assumed the moments condition above for all $|\beta| \leq 2M$.  In particular, note that $\gamma > \frac{4}{d}$, hence these terms have at least quadratic decay in $t$ of the $L^\infty$ norm.

We also have
\begin{eqnarray*}
|a| = O(R^{p-1} |v_0| + O (|v_0|^p)),
\end{eqnarray*}
hence
\begin{eqnarray*}
| D^\alpha a(x,t) | & \leq & \frac{C}{t^{\frac{d}{2} p}}, \\
|e^{-c|x|} D^\alpha a(x,t)| & \leq & \frac{C}{t^{M+\frac{d}{2}}}.
\end{eqnarray*}

Similarly,
\begin{eqnarray*}
| D^\alpha b(x,t) | & \leq & \frac{C}{t^{\frac{d}{2} p}}, \\
|e^{-c|x|} D^\alpha b(x,t) | & \leq & \frac{C}{t^{M+\frac{d}{2}}}.
\end{eqnarray*}

Now, we look at the integral formulation of the equation for $w$
\begin{eqnarray}
\label{eqn:contw1}
w(t) & = & -i \int_t^\infty e^{i(\tau-t) \mathcal{H}} P_d [f_0 + aw + b \bar{w} + G(w)](\tau) d\tau \\
\label{eqn:contw2}
& + & -i \int_t^\infty e^{i(\tau-t) \mathcal{H}} P_c [f_0 + aw + b \bar{w} + G(w)](\tau) d\tau,
\end{eqnarray}
where $P_S$ projects onto the singular part of the spectrum and $P_M$ projects onto the discrete part of the spectrum.  

Since we are interested in minimal regularity perturbations, we first want to see that 
\begin{eqnarray*}
\| w(t) \|_{H^2} \leq \frac{1}{t^{N_1}},
\end{eqnarray*}
for $t \geq \frac{1}{\delta}$ and $N_1$ to be determined.


To do this, we will discuss Equations \eqref{eqn:contw1} and \eqref{eqn:contw2} separately.

From the following Corollary in \cite{Mlin},
\begin{eqnarray*}
\| P_{c} e^{i t \mathcal{H}} f \|_{L^2} \lesssim \| f \|_{L^2},
\end{eqnarray*}
we have
\begin{eqnarray*}
\| \text{\eqref{eqn:contw1}} \|_{H^2} & \leq & \int_t^\infty [1 + (\tau-t)^3] \left\{ \int |f_0 + aw + b \bar{w} + G(w)](x,\tau) | e^{-c|x|} dx \right\} dt \\
& \leq & \int_t^\infty [1 + (\tau-t)^3] [\frac{C}{\tau^{M+\frac{d}{2}}} + \frac{C}{\tau^{M+\frac{d}{2}}} \| w(\tau) \|_{H^2} + C \| w(\tau) \|_{H^2}^2 ] d\tau.
\end{eqnarray*}
Hence, by assuming $M$, $N_1$ large enough and using a bootstrapping argument
\begin{eqnarray*}
\| \text{\eqref{eqn:contw1}} \|_{H^2} \leq \int_t^\infty [1 + (\tau-t)^3] \left[  \frac{C}{\tau^{2 N_1}} \right] d\tau < \frac{C}{t^{2N_1-4}} < \frac{C \delta}{t^{N_1}}.
\end{eqnarray*}

For the second part of this argument, we see
\begin{eqnarray*}
\| \text{\eqref{eqn:contw2}} \|_{H^2} & \leq & \int_t^\infty \| f_0 + aw + b \bar{w} + G(w)](\tau)\|_{H^2} d\tau \\
& \leq & \int_t^\infty \left\{ \frac{C}{\tau^{\frac{d}{2} p}} + \frac{C}{\tau^{\frac{d}{2} p}} \| w(\tau) \|_{H^2} + C \| w(\tau) \|_{H^2}^2 \right\} d\tau \\
& \leq & \int_t^\infty \left\{ \frac{C}{\tau^{\frac{d}{2} p}} + \frac{C}{\tau^{N_1 + \frac{d}{2} p}} + \frac{C}{\tau^{2N_1}} \right\} d\tau \\
& \leq & \frac{C}{t^{N_1+1}} \leq \frac{C \delta}{t^{N_1}},
\end{eqnarray*}
provided $\frac{d}{2} p$ is large enough.    

We are also be able to show 
\begin{eqnarray*}
\| w \|_{L^2( |x|^A dx)} \leq \frac{C}{t^{N_2}}
\end{eqnarray*}  
for $t > \frac{1}{\delta}$ and $N_2$ to be determined.  Then, we will have the desired contraction argument for the linear perturbation.  

From the necessary dispersive estimate given by
\begin{eqnarray*}
\| |x|^\alpha e^{itH} (P_S \phi) \|_{L^2} \leq C (1 + |t|^3) \int |\phi| e^{-c|x|} dx,
\end{eqnarray*}
the estimate for Equation \eqref{eqn:contw1} follows immediately from the $H^s$ argument.  

For Equation \eqref{eqn:contw2}, we need the following estimate
\begin{eqnarray*} 
\| |x|^\alpha e^{itH} (P_M \phi) \|_{L^2} \leq C \| |x|^\alpha \phi \|_{L^2} + C ( 1 + |t|^\alpha ) \| \phi \|_{H^\alpha}.
\end{eqnarray*}

Then,
\begin{eqnarray}
\label{eqn:contw3}
\| \text{\eqref{eqn:contw2}} \|_{L^2(|x|^A dx)} & \leq & C \int_t^\infty \| [f_0 + aw + b\bar{w} + G(w)](\tau) \|_{L^2(|x|^A dx)} d\tau \\
\label{eqn:contw4}
&+& \int_t^\infty (1+|t-\tau|^A) \| [f_0 + aw + b\bar{w} + G(w)](\tau) \|_{H^A} d\tau.
\end{eqnarray}

Again, we look at each integral separately.  For Equation \eqref{eqn:contw3},
\begin{eqnarray*}
\text{\eqref{eqn:contw3}} & \leq & \int_t^\infty \left\{ \frac{C}{\tau^{2N}} + \frac{C}{\tau^2} \|w(\tau)\|_{L^2(|x|^A dx)} + C \| w(\tau) \|_{L^\infty} \|w(\tau)\|_{L^2(|x|^A dx)} \right\} d\tau \\
& < & \left\{ \frac{C}{\tau^{2N}} + \frac{C}{\tau^{2+N1}} + \frac{C}{\tau^{N+N1}} \right \} d\tau \\
& < & \frac{C}{t^{N1+1}} < \frac{C \delta}{t^{N1}}.
\end{eqnarray*}

For \eqref{eqn:contw4}, 
\begin{eqnarray*}
 \text{\eqref{eqn:contw4}} & \leq & C \int_t^\infty [(1+(\tau-t)^A] \left\{ \frac{C}{\tau^{\frac{d}{2} p + A}} + C \| w(\tau)\|^2_{H^A} \right\} d\tau \\
& \leq & C \int_t^\infty \tau^A \left\{ \frac{C}{\tau^{2N_2}} + \frac{C}{\tau^{\frac{d}{2} p}} \right\} d\tau \\
& \leq & \frac{C}{t^{2 N_2-A-1}} < \frac{C \delta}{t^{N_2}},
\end{eqnarray*}
for $N_2$, $p$ sufficiently large.  

Hence, the contraction argument goes through and we have the desired bound on $\| w\|_X$.  



\section{Optimization for LS}
\label{thm1:s-opt}

We seek optimal values for the spaces and decay in the case where the perturbation solves the linear Schrodinger equation.  

To begin, allow $A$ and $M$ to be arbitrary for now and we will select them later.  Assume that $\phi \in X_A$ and that the first $2M$ moments of $\phi$ vanish.  By writing the linear solution in integral form, we see that
\begin{eqnarray*}
\| D^\alpha v_0 \|_{L^\infty} & \leq & \frac{C}{t^{\frac{d}{2}}}, \\
|e^{-c|x|} D^\alpha v_0 (x,t)| & \leq & \frac{C}{t^{\frac{d}{2}+M}},
\end{eqnarray*}
for $\alpha < M$.  

To gain in time decay for the linear Schr\"odinger equation, we make the assumption that 
\begin{eqnarray}
\label{eqn:moments1}
\int x^\alpha \phi(x) dx = 0,
\end{eqnarray}
where $\alpha$ is a multi-index where $|\alpha| \leq M$ for $M$ to be determined below.  Note that the function space $\mathcal{P}^A_1$ in Section \ref{thm:mainLS} is determined by functions $\phi \in X^A$ coupled with taking moments conditions for $|\alpha| \leq A$.

\begin{lem}
\label{lem1LS:thm1}
Since $R \in \mathcal{S}$,
\begin{eqnarray*}
\| (Rv_0)(t) \|_{H^s} \leq \frac{C}{t^M}.
\end{eqnarray*}
\end{lem}

\begin{proof}

We have
\begin{eqnarray*}
\| (R v_0) (t) \|_{L^2} \leq \| \langle x \rangle^{-N} v_0 \|_{L^\infty} \| \langle x \rangle^N R \|_{L^\infty}.
\end{eqnarray*}
Hence, using the principal of nonstationary phase away from the origin and the moments condition near the origin in the fundamental solution for linear Schr\"odinger, we gain in time decay.  Note that in order to gain in time decay away from the origin, it is essential that we have the weight in order to control all of the terms resulting from integrating by parts.  The higher derivative terms follow similarly.

\end{proof}

\begin{lem}
\label{lem2LS:thm1}
For $f_0$ described above for $v_0 = e^{i \Delta t} \phi$, we have
\begin{eqnarray*}
\| f_0 \|_{H^2} \leq \frac{C}{t^{\frac{d}{2} p}}
\end{eqnarray*}   
and
\begin{eqnarray*}
\| e^{-c|x|} f_0 \|_{L^\infty} \leq \frac{C}{t^{(\frac{d}{2}+M) p}}.
\end{eqnarray*}
\end{lem}

\begin{proof}

The $O(R v_0)$ term is controlled by similar analysis to that in \ref{lem1LS:thm1}.  Hence, we concern ourselves with the $O(|v_0|^{p+1})$ term.  To that end, we have
\begin{eqnarray*}
\| v_0^{p+1} \|_{L^2} \leq \|  v_0 \|^p_{L^\infty} \| v_0 \|_{L^2}.
\end{eqnarray*}
Consequently, we have
\begin{eqnarray*}
\| v_0^{p+1} \|_{L^2} \lesssim \langle t \rangle^{-\frac{d}{2} p}.
\end{eqnarray*}

\end{proof}

Once again, since the decay rate is determined by the number of moments for the linear equation,
\begin{eqnarray*}
|e^{-c|x|} D^\alpha a(x,t)| & \leq & \frac{C}{t^{M+\frac{d}{2}}}, \\
|e^{-c|x|} D^\alpha b(x,t)| & \leq & \frac{C}{t^{M+\frac{d}{2}}}.
\end{eqnarray*}

As we desire to work with low regularity perturbations, let us simply assume that $\| w \|_{H^2}^2 \lesssim t^{-N}$.
Now, we must choose $A$ and $M$ optimally for the contraction argument to work.  From Equation \eqref{eqn:contw1}, we require that 
\begin{eqnarray*}
2N-4 & \geq & N, \\
\frac{d}{2} + M  - 4 & \geq & N,
\end{eqnarray*}
where the moments condition is determined by the $O(v_0 R^p)$ term.  So, we gather that $N > 4$ and $M > 8- \frac{d}{2}$.  The number of moments necessary will depend upon the the dimension $d$.  In $\reals^3$, we have $M > \frac{13}{2}$.  

From Equation \eqref{eqn:contw2}, we have only one more requirement
\begin{eqnarray*}
\frac{d}{2} p - 1 > N.
\end{eqnarray*}
At this stage, we see that given $N > 4$, we need $p > \frac{10}{3}$.  Clearly, the restrictions on $p$ lessen as $d$ gets large.  In particular, we cannot show the existence of stable perturbations for minimal mass solitons of NLS equations with nonlinearities of type $2$ in $\reals^3$.  A variation of this argument will be explored later to show long time stability under restricted perturbations.

\section{Linearization Scheme for $\mathcal{H}$-LS perturbations}
\label{thm1:lin}

Again, let 
\begin{eqnarray*}
u = R e^{it} + z_\phi e^{it} + w e^{it},
\end{eqnarray*}
except now we have
\begin{eqnarray}
\label{eqn:HLS}
i z_t + \mathcal{H} z & = & 0, \\
z(0,x) & = & \phi(x),
\end{eqnarray}
where $\mathcal{H}$ is linear operator resulting from linearizing about the minimal mass soliton.   We refer to Equation \eqref{eqn:HLS} as the $\mathcal{H}$linear Schr\"odinger equation ($\mathcal{H}$-LS).

Now, let $v_0 = z_\phi e^{-it}$.  Again, we have the same equation,
\begin{eqnarray*}
i v_t + \Delta v - v + F(|R+v|^2)(R+v) - F(R^2)R = 0,
\end{eqnarray*}
where $u(x,t) = e^{it} (R + v)$.

However, since 
\begin{eqnarray*}
i (v_0)_t + \Delta v_0 - v_0 + [F(R^2) + F'(R^2)R^2]v_0 + F'(R^2)R^2 \overline{v_0} = 0,
\end{eqnarray*}
we have 
\begin{eqnarray*}
i w_t & + &\Delta w - w + [F(|R+v_0+w|^2)(R+v_0+w) - F(R^2)R - (F(R^2)+F'(R^2)R^2)v_0\\
& - & F'(R^2)R^2\overline{v_0}] = 0.
\end{eqnarray*}

Hence, let
\begin{eqnarray*}
f_0 & = & F(|R+v_0|^2)(R+v_0) - F(R^2)R - (F(R^2)+F'(R^2)R^2)v_0 - F'(R^2)R^2\overline{v_0}, \\
a & = & [F(|R+v_0|^2) + F'(|R+v_0|^2)|R+v_0|^2] - [F(R^2) + F'(R^2)R^2],  \\
b & = &  F'(|R+v_0|^2)(R+v_0)^2 - F'(R^2) R^2, \\
G(w) & = & F(|R+v_0+w|^2) (R+v_0+w) - F(|R + v_0|^2)(R+v_0) \\
& - & [F(|R+v_0|^2) + F'(|R+v_0|^2)|R+v_0|^2] w - F'(|R+v_0|^2)(R+v_0)^2 \bar{w}.
\end{eqnarray*}

Hence, we now have
\begin{eqnarray*}
|f_0| & \lesssim & O(R^p |v_0|^2) + O(|v_0|^{p+1})| , \\
|a| & \lesssim & O(R^{p-1} |v_0|) + O(|v_0|^{p}), \\
|b| & \lesssim & O(R^{p-1} |v_0|) + O(|v_0|^{p}) \\
|G| & \lesssim & O( |w|^2 ).
\end{eqnarray*}
Notice that since the linear terms in $v_0$ have been removed from $f_0$, we expect to require fewer moments conditions for the contraction argument to hold.

\section{Optimization for $\mathcal{H}$-LS}
\label{thm1:Hopt}

In the scheme where we solve the linear perturbation using $\mathcal{H}$, we
have now introduced gain in the $f_0$ term.  Specifically, we now have
\begin{eqnarray*}
|f_0| \lesssim O(R |v_0|^2) + O(|v_0|^{p+1}).
\end{eqnarray*}  

For the $\mathcal{H}$-LS case, we use the moment conditions derived in \cite{Mlin} in order to gain decay in time locally in space.  Note that the function space $\mathcal{P}^A_2$ in Theorem \ref{thm:mainHLS} is determined by functions $\phi \in X^A$ coupled with taking moments conditions for $|\beta| \leq A$.

\begin{lem}
\label{lem1HLS:thm1}
Since $R \in \mathcal{S}$, if $v_0$ satisfies the first $M$ moments conditions from \cite{Mlin}
\begin{eqnarray*}
\| (Rv_0)(t) \|_{H^s} \leq \frac{C}{t^{\frac{d}{2}+M}}.
\end{eqnarray*}
\end{lem}

\begin{proof}

We have
\begin{eqnarray*}
\| (R v_0) (t) \|_{L^2} \leq \| \langle x \rangle^{-N} v_0 \|_{L^\infty} \| \langle x \rangle^N R \|_{L^\infty}.
\end{eqnarray*}
Hence, using the principal of nonstationary phase away from the origin and the moments condition near the origin on
\begin{eqnarray}
\label{eqn:intrep}
e^{i \mathcal{H} t} P_c \phi = Q^{-1} e^{i t W} Q \phi,
\end{eqnarray}
as derived in \cite{Mlin}, we gain in time decay.  Note that in order to gain in time decay away from the origin, it is essential that we have the weight in order to control all of the terms resulting from integrating by parts.  The higher derivative terms follow similarly.

\end{proof}

\begin{lem}
\label{lem2HLS:thm1}
For $f_0$ described above for $v_0 = e^{i \mathcal{H} t} \phi$, we have
\begin{eqnarray*}
\| f_0 \|_{X_A} & \leq  & \frac{C}{t^{\frac{d}{2} p}},
\end{eqnarray*}
for $s<M$.   
\end{lem}

\begin{proof}

The $O(R v_0)$ term is controlled by similar analysis to that in \ref{lem1HLS:thm1}.  Hence, we concern ourselves with the $O(|v_0|^{p+1})$ term.  To that end, we have
\begin{eqnarray*}
\| v_0^{p+1} \|_{L^2} \leq \| v_0 \|^p_{L^\infty} \| v_0 \|_{L^2}.
\end{eqnarray*}
Using a similar analysis from Lemma \ref{lem2LS:thm1} on Equation \eqref{eqn:intrep}, we have
\begin{eqnarray*}
\| v_0^{p+1} \|_{L^2} \lesssim \langle t \rangle^{\frac{d}{2} p}
\end{eqnarray*}
using the $L^2$ boundedness results for $e^{i \mathcal{H} t}$.

\end{proof}

Now, since we are dealing with nonlinearities with minimal smoothness, we wish to run the contraction argument with minimal assumptions on $w(x,t)$.  Then, we assume
\begin{eqnarray*}
\| w(x,t) \|_{L^2}  <  \frac{1}{t^{N}},
\end{eqnarray*}
for some $N$ to be determined.  Then,
\begin{eqnarray}
\label{eqn:w1}
w(t) & = & -i \int_t^\infty e^{i(\tau-t)H} P_d [f_0 + aw + b \bar{w} + G(w)](\tau) d\tau \\
\label{eqn:w2}
& + & -i \int_t^\infty e^{i(\tau-t)H} P_c [f_0 + aw + b \bar{w} + G(w)](\tau) d\tau,
\end{eqnarray}
where $P_d$ projects onto the discrete part of the spectrum and $P_c$ projects onto the continuous part of the spectrum.  So, using the dispersive estimates, we have
\begin{eqnarray*}
\| \text{\eqref{eqn:w1}} \|_{H^2} & \leq & \| \int_t^\infty e^{i(\tau-t)H} P_S [f_0 + aw + b \bar{w} + G(w)](\tau) d\tau \|_{H^2} \\
& \leq & \int_t^\infty [1 + (\tau - t)^3] \left\{ \int | f_0 + aw + b \bar{w} G(w)](x,t) e^{-c|x|} dx \right\} d\tau \\
& \leq & \int_t^\infty [1 + (\tau - t)^3] \left\{ \int | f_0 + aw + b \bar{w} G(w)](x,t) e^{-c|x|} dx \right\} d\tau \\
& \lesssim & \int_t^\infty [1 + (\tau - t)^3] \left\{ \frac{1}{\tau^{2N}} + \frac{1}{\tau^N} \| w(\tau) \|_{L^2} + \| w(\tau) \|_{L^2}^2 \right\} d\tau \\
& \lesssim & \frac{1}{t^{(\frac{d}{2}+M)2-4}} + \frac{1}{t^{(\frac{d}{2}+M)p+N-4}} + \frac{1}{t^{2N-4}} , \\
\| \text{\eqref{eqn:w2}} \|_{L^2} & \leq & \| \int_t^\infty e^{i(\tau-t)H} P_M [f_0 + aw + b \bar{w} + G(w)](\tau) d\tau \|_{L^2} \\
 & \leq & \int_t^\infty \| \int | f_0 + aw + b \bar{w} G(w)](x,t) e^{-c|x|} dx \|_{L^2} d\tau \\
& \leq & \left\{ \frac{1}{\tau^{\frac{d}{2} p}} + \frac{1}{\tau^N} \| w(\tau) \|_{H^2} + \| w(\tau) \|_{H^2}^2 \right\} \\
& \lesssim & \frac{1}{t^{2N-1}} + \frac{1}{t^{M+N-1}} + \frac{1}{t^{2M-1}}.
\end{eqnarray*}
Hence, we require once again that that $2N-4 \geq N$, but the moments condition is determined now by the $O( v_0^2 R^{p-1})$ term, so we have
\begin{eqnarray*}
d + 2M -4 > 4.
\end{eqnarray*}
In $\reals^3$ that $M > \frac{5}{2}$, or $M > 2$. The condition on $p$ however does not change whatsoever, therefore are again only considering nonlinearities of type $1$ with $p > \frac{10}{3}$.


\section{Manifolds of Perturbations}
\label{thm1:man}

From Theorems \ref{thm:mainLS} and \ref{thm:mainHLS}, we would like to know that our perturbative solution actually lives on a finite codimension submanifold.  Specifically, given spaces $X_1$, $X_2$ with $X_1 \subset X_2$ and norms $\| \cdot \|_{X_1}$, $\| \cdot \|_{X_2}$ respectively, we require a finite codimension subset $\mathcal{S} \subset X_1$ and a map $\Psi : \mathcal{B} \cap \mathcal{S} \to X_2$ where
\begin{eqnarray*}
\mathcal{B} = \left\{ \phi \in X_1 | \| \phi \|_{X_1} < \delta \right\}.
\end{eqnarray*}
For Theorems \ref{thm:mainLS} and \ref{thm:mainHLS} above, we have 
\begin{eqnarray}
\label{eqn:psi}
\Psi ( \phi) = w(t_0).
\end{eqnarray}

\begin{lem}
\label{lem:man}
For the map $\Psi$ defined by \eqref{eqn:psi}, we have
\begin{eqnarray}
\label{eqn:man1}
\| \Psi (\phi) \|_{X_2} & \lesssim & \| \phi \|_{X_1}^2, \ \phi \in \mathcal{B} \cap \mathcal{S}, \\
\label{eqn:man2}
\| \Psi (\phi_1) - \Psi (\phi_2) \|_{X_2} & \lesssim & \delta \| \phi_1 - \phi_2 \|_{X_1}, \  \phi_1, \phi_2 \in \mathcal{B} \cap \mathcal{S},
\end{eqnarray}
where 
\begin{eqnarray}
\label{space:X1}
X_1 & = & L^2 (|x|^{3+} dx) \cap H^2, \\
\label{space:X2}
X_2 & = & H^2,
\end{eqnarray}
and 
\begin{eqnarray}
\mathcal{S} = \{ \phi \in H^2 | \phi = P_c \phi, \ \phi \in \mathcal{P}_2^A \}
\end{eqnarray}
for some $A > 2$.
\end{lem}

\begin{rem}
In Lemma \ref{lem:man}, Equation \eqref{eqn:man1} shows that the tangent space at $0$ of the stable submanifold, $\mathcal{M}$, is the space $\mathcal{S}$, while Equation \eqref{eqn:man2} shows that $\mathcal{M}$ is given by a Lipschitz parametrization.  
\end{rem}

\begin{rem}
The codimension of $\mathcal{S}$ will be at most $2d+4$ since $H^1 \times H^1 = N_g(\mathcal{H}) \oplus \{ N_g(\mathcal{H}^*) \}^\perp$ and $N_g (\mathcal{H}) = 2d + 4$.  It is possible that the size of $\mathcal{S}$ can be improved beyond this codimension, which the author will explore in future work. 
\end{rem}  

\begin{proof}[Proof of Lemma \ref{lem:man}]

Let us first prove \eqref{eqn:man1}.  Assume that $ \| w \|_{X_2} \leq \| \phi \|_{X_1}^2$.  Then,
\begin{eqnarray*}
\| w \|_{X_2} & \leq & \int_{t_0}^\infty(1 + (t_0 - \tau)^3 )  \left[ \int v_0^2 e^{-c|x|} dx + \int v_0^2 e^{-c|x|} dx + \int v_0^2 e^{-c|x|} dx \right] d\tau \\
& + & \int_{t_0}^\infty  \left[ \| f_0 \|_{X_2} + \| a w \|_{X_2} + \| b \bar{w} \|_{X_2} + \| G(w) \|_{X_2} \right] d\tau \\
& \lesssim & t_0^{- \epsilon} \| \phi \|^2_{X_1}
\end{eqnarray*}
using our assumptions as well as the decay of $\| w \|_{H^2}$ from the proof of Theorem \ref{thm:mainHLS}.  Hence, taking $t_0$ to be large, the result follows.

Now, for \eqref{eqn:man2}, we have
\begin{eqnarray*}
w_1 - w_2 = \int_{t_0}^\infty e^{i \mathcal{H} (\tau - t_0)} [ (f_0^1 - f_0^2) + a (w_1 - w_2) + b (\bar{w}_1 - \bar{w}_2) + (G(w_1) - G(w_2)) ] d\tau.
\end{eqnarray*}
Since
\begin{eqnarray*}
|G(w_1) - G(w_2)| \sim |w_1 + w_2| |w_1 - w_2|
\end{eqnarray*}
and
\begin{eqnarray*}
|f_0^1 - f_0^2| \sim R |\phi_1 + \phi_2| |\phi_1 - \phi_2| + (|\phi_1|^p + |\phi_2|^p ) |\phi_1 - \phi_2|,
\end{eqnarray*}
the result follows from a similar continuity argument to that above using \eqref{eqn:man1}.  
\end{proof}

\begin{rem}
Note that for $p$ large enough in type $1$ nonlinearities, using the dispersive estimates $(iv)$, $(v)$ from Theorem \ref{thm:disp1} and the fact that 
\begin{eqnarray*}
\| \phi \|_{L^1} \leq C(\epsilon) \| \phi \langle x \rangle^{d+\epsilon} \|,
\end{eqnarray*}
we can take $X_1 = X_2 = X^{d+}$ to have a true manifold of perturbations.
\end{rem}

\section{Long Time Analysis for Type $2$ Nonlinearities}
\label{thm1:longtime}

For NLS with saturated nonlinearities of type $2$, we can no longer do the global scattering analysis from above.  Instead, we have Theorem \ref{thm:longtime}:

\begin{thm}
Given Equation \eqref{eqn:NLS}, where $\beta$ is an admissible saturated nonlinearity of type $2$, for any $\phi = P_c \phi \in W^{2,1} \cap H^2$ with $\| \phi \|_{W^{2,1} \cap H^2} < \delta < 1$, Equation \eqref{eqn:NLS} has a solution for $t \in [0,\left( 2 \delta \right)^{-\frac{1}{4}} )$ of the form 
\begin{eqnarray*}
u(x,t) = R_{min} + v(t) = R_{min} + e^{i \mathcal{H} t} \phi + w(x,t),
\end{eqnarray*}
where 
\begin{eqnarray*}
u(x,0) = R_{min} + \phi.
\end{eqnarray*} 
\end{thm}

\begin{proof}
Instead of the scattering point of view, we look at solving for the perturbation forward in time.  Namely, we have
\begin{eqnarray*}
w(t) & = & i \int_0^t e^{ i \mathcal{H} (t-\tau)} [f_0 + aw + b \bar{w} + G(w)](\tau) d\tau \\
& = & i \int_0^t e^{ i \mathcal{H} (t-\tau)} P_d [f_0 + aw + b \bar{w} + G(w)](\tau) d\tau \\
& + & i \int_0^t e^{ i \mathcal{H} (t-\tau)} P_c [f_0 + aw + b \bar{w} + G(w)](\tau) d\tau.
\end{eqnarray*}

Let us assume that \begin{eqnarray}
\| w \|_{L^\infty [0,T] H^2} & < & \delta.
\end{eqnarray}
Then
\begin{eqnarray*}
\| w \|_{L^\infty [0,T] H^2} & \lesssim & \left\| \int_0^t [1 + (t-\tau)^3] \left( \int f_0 e^{-c|x|} dx + \int a w e^{-c|x|} dx \right. \right. \\
& + & \left. \int b \bar{w} e^{-c|x|} dx + \int G(w) e^{-c|x|} dx \right) d \tau  \\
& + & \left. \int_0^t [ \| f_0 \|_{H^2} + \| a w \|_{H^2} + \| b \bar{w} \|_{H^2} + \| G(w) \|_{H^2} ] d \tau \right\|_{L^\infty_t [0,T]} \\
& \lesssim & \left\| \int_0^t [1 + (t-\tau)^3] \left( \langle \tau \rangle^{-3(\frac{d}{2})} \| \phi \|_{L^1}^3  \right. \right.\\
& + & \left. \langle \tau \rangle^{-(\frac{d}{2})} \| w \|_{H^2} \| \phi \|_{L^1} + \| w \|_{H^2}^2 \right) d \tau  \\
& + & \left. \int_0^t [ \langle \tau \rangle^{-d} \| \phi \|_{L^1}^2 \| \phi \|_{H^2} + \langle \tau \rangle^{-\frac{d}{2}} \| \phi \|_{L^1} \| w \|_{H^2} + \| w \|^2_{H^2} ] d \tau \right\|_{L^\infty_t [0,T]} \\
& \lesssim & T^4 \| w \|^2_{L^\infty [0,T] H^2} + T^{\frac{5}{2}} \| \phi \|_{L^1}^2 \| w \|_{L^\infty [0,T] H^2} + T^{-\frac{1}{2}} \| \phi \|^3_{L^1} \\
& \lesssim & \frac{\delta}{2} \| w \|_{L^\infty [0,T] H^2} + \delta^{3+\frac{1}{8}}.
\end{eqnarray*}
Hence, by a continuity argument, $w$ exists on $[0,T]$ with $\| w \|_{L^\infty [0,T] H^2} \leq \delta^{3+\frac{1}{8}}$.
\end{proof}

\begin{rem}
Since we assume $\| \phi \|_{L^1} \leq \delta$, there exists a $\delta_0$ independent of $\phi$ such that after each time step $T$, the initial perturbation $\phi(T)$ is of size $\delta < \delta_0$ in $W^{2,1} \cap H^2$.  Hence, we can continue this iteration from $T$ to $NT$ for any $N > 0$.  Thus, there exists a small perturbative solution of the minimal mass soliton, however we are unable to show the lower order terms disperse as $t \to \infty$.  If by using the ideal scaling of $\| w \| \leq \delta^{3+}$, then we see we have existence on a time scale of $T \approx \delta^{-\frac{3}{4}-}$ and $\| w \| \lesssim \delta^{3 + \frac{3}{8} + }$.  Once again, due to the gain in $\| w \|$, this argument can be continued for all $t$, though we cannot prove dispersion of $w$.
\end{rem}

\end{document}